\documentclass[11pt,twoside,a4paper]{article}
\usepackage{amsmath,amssymb,amsthm}

\usepackage{graphicx,psfrag}

\newtheorem{theorem}{Theorem}[section]
\newtheorem{lemma}[theorem]{Lemma}
\newtheorem{proposition}[theorem]{Proposition}
\newtheorem{corollary}[theorem]{Corollary}

\theoremstyle{definition}

\theoremstyle{remark}
\newtheorem{remark}[theorem]{Remark}
\newtheorem{definition}[theorem]{Definition}

\newcommand{\R}{\mathbb{R}}
\newcommand{\Z}{\mathbb{Z}}
\newcommand{\N}{\mathbb{N}}
\newcommand{\Hy}{\mathbb{H}}

\newcommand{\cA}{\mathcal{A}}

\newcommand{\cE}{\mathcal{E}}

\newcommand{\al}{\alpha}

\newcommand{\ga}{\gamma}
\newcommand{\Ga}{\Gamma}

\newcommand{\De}{\Delta}
\newcommand{\ep}{\varepsilon}
\newcommand{\om}{\omega}

\newcommand{\la}{\lambda}

\renewcommand{\phi}{\varphi}

\newcommand{\rank}{\operatorname{rank}}

\newcommand{\CAT}{\operatorname{CAT}}

\newcommand{\id}{\operatorname{id}}

\newcommand{\ISO}{\operatorname{ISO}}
\newcommand{\Ax}{\operatorname{Ax}}

\newcommand{\crt}{\operatorname{crt}}

\renewcommand{\d}{\partial}

\newcommand{\sub}{\subset}

\begin{document}

\title{Group actions on geodesic Ptolemy spaces}
\author{Thomas Foertsch,  
 \  Viktor Schroeder\footnote{Both authors are supported by the Swiss National
Science Foundation, grant 200021-115919}}

\maketitle

\begin{abstract}
In this paper we study geodesic Ptolemy metric spaces $X$ which allow
proper and cocompact isometric actions of crystallographic or, 
more generally, virtual polycyclic groups.
We show that $X$ is equivariantly rough isometric to a Euclidean space.
\end{abstract}


\section{Introduction}

\subsection{Statement of Results}

A metric space $X$ is called a Ptolemy metric space,
if the inequality
\begin{equation} \label{eqn:ptolemy}
|xy||uv| \; \le \; |xu| \, |yv| \; + \; |xv| \, |yu|
\end{equation}
is satisfied for all $x,y,u,v\in X$. \\

In this paper we study 
subgroups $\Ga \subset \ISO(X)$ of the isometry group
of geodesic Ptolemy spaces
$X$, which act properly and cocompactly on $X$. 
Our motivation to study these spaces is described in
Subsection \ref{subsec:motivation}. \\

We adopt the notation of \cite{BH}(p.131) and  call an action
{\em proper}, if for every $x\in X$ there is a number $r>0$ such that
the set $\{\ga\in \Ga \mid |x\,\ga(x)| < r\}$ is finite. An action
is {\em cocompact}, if there is a compact subset $K \subset X$, such that
$X = \Ga K$.
Since $X$ allows a proper and cocompact action by isometries,
$X$ is complete and, since $X$ is geodesic, it is indeed a proper 
metric space, i.e. closed distance balls are compact
(see \cite{BH} p. 132, 8.4(1)).\\
The fact that $X$ is a proper metric space is used
essentially in our paper.

We recall that a group $\Ga$ is called {\em crystallographic} if, for some
$n\in \N$, $\Ga$ is isomorphic
(as an abstract group) to
a discrete, cocompact subgroup of the isometry group
$\ISO(\R^n)$. The number $n$ is then called the {\em rank} of $\Ga$.
A map $\phi:X\to Y$ between metric spaces is called  
{\em roughly isometric}, if there exists a constant $A \ge 0$ such that for all $x,x'\in X$ 
we have
\begin{equation} \label{eqn:roughisometric}
|xx'| -A \le |\phi(x)\phi(x')|\le |xx'|+A.
\end{equation}
If, in addition, for all $y \in Y$ there exits $x\in X$ with $|y\,\phi(x)|\le A$, we call
$\phi$ a {\em rough isometry}.
 
\begin{theorem} \label{thm:cryst}
Let $X$ be a geodesic Ptolemy metric space and
let $\Ga \subset \ISO(X)$ be a
crystallographic group of rank $n$ acting properly and cocompactly on $X$.
Then there exists a continuous rough isometry
$\phi:X \to \R^n$, which is equivariant with respect to
the group operation, i.e.
$\phi(\ga x)=\ga \phi(x)$, where $\Ga$ acts on $\R^n$ by isometries with compact quotient.
\end{theorem}

In the case that $X$ is a topological manifold we can say more.

\begin{theorem} \label{thm:manifold}
Let $X$ be a geodesic Ptolemy metric space
which is in addition a topological manifold. 
Assume that $\Ga \subset \ISO(X)$ 
is a crystallographic group of rank $n$
which acts properly and cocompactly on $X$.
Then there actually exists an equivariant isometry
$\phi:X \to \R^n$.
\end{theorem}

The methods of the proof in connection with the
results in \cite{CS} give also the following general structure result
for abelian groups operating on Ptolemy spaces:

\begin{theorem} \label{thm:straight}
Let $X$ be a geodesic Ptolemy metric space and
let $\Ga \subset \ISO(X)$ act properly and cocompactly on $X$.
Then every free abelian subgroup of $\Ga$ is straight in $\Ga$.
\end{theorem}

To define the notion of straightness we recall the word norm
$|\ |_{\Ga}$ of a finitely generated group. Let $\{\ga_1,\ldots,\ga_k\}$ be
a set of generators of $\Ga$. Then for $\ga \in \Ga$, $|\ga|_{\Ga}$ is
defined as the length of the shortest word in the $\ga_i$ and $\ga_i^{-1}$
representing $\ga$. This norm depends on the choice of generators. However,
as it is easy to see, different sets of generators give rise to equivalent
norms. That is, if $|\ |^1_{\Ga}$ and $|\ |^2_{\Ga}$ are such norms,
then there exists a constant $c\ge 1$ such that
$(1/c)|\ga |^1_{\Ga} \le |\ga |^2_{\Ga}\leq c|\ga|^1_{\Ga}$ for all $\ga\in \Ga$.
A finitely generated subgroup $\Ga_0\subset \Ga$ is called
{\it straight} in $\Ga$, if $|\ |_{\Ga_0}$ and $|\ |_{\Ga}$ are equivalent norms on
$\Ga_0$. This notion is independent of the choice of generators on $\Ga_0$ and
$\Ga$.

As a consequence of Theorem \ref{thm:straight} we 
can generalize our result to the action of virtual polycyclic groups.
A group is called {\em polycyclic}, if it admits a finite chain of
subgroups $G_0\subset G_1\subset \cdots G_l$, such that
$G_0$ is the trivial subgroup and $G_l=G$, such that
$G_i/G_{i-1}$ is cyclic. Equivalently a group is polycyclic if
it is solvable and every subgroup is finitely generated.
A group is called {\em virtual polycyclic} if it contains a polycyclic
subgroup of finite index.
In particular crystallographic groups are virtual polycyclic.

\begin{corollary}\label{cor:solv}
Theorem \ref{thm:cryst} and Theorem \ref{thm:manifold}
hold also if $\Ga \subset \ISO(X)$ is assumed to be virtual
polycyclic.
In particular if a virtual polycyclic group acts properly and
cocompactly on a geodesic Ptolemy space then 
there exists a homomorhism
$\Ga \to \ISO(\R^n)$ with finite kernel such that
the image
is a crystallographic group. 
\end{corollary}

At the end of the introduction we state a result, which we need in the proof
and which is of independent interest, see Section \ref{sec:convexity}.

\begin{theorem} \label{thm:sdc}
A proper, geodesic Ptolemy metric space is strictly distance convex.
\end{theorem}

It is a pleasure to thank Alexander Lytchak for many inspiring discussions.


\subsection{Outline of the main ideas, Ptolemy versus $\CAT(0)$} \label{subsec:outline}

All the above results are well known if the space $X$ is a $\CAT(0)$
space. A $\CAT(0)$-space is Ptolemy (see \cite{FLS}).
It is an open question if proper Ptolemy metric spaces are
$\CAT(0)$ (we remark that there are nonlocally compact Ptolemy spaces which are not $\CAT(0)$,
see \cite{FLS}).
To be very concrete, we do not even know, whether the flat strip theorem of $\CAT(0)$-spaces
(see e.g. \cite{BH} p.182) holds for Ptolemy metric spaces:
Assume that $d$ is a distance function on the set $X=\R\times [0,1]$, such that $(X,d)$ is a geodesic
Ptolemy space, the topology on $(X,d)$ 
coincides with the standard topology, the curves $t\mapsto (t,i)$ are
geodesics parameterized by arclength for $i=0,1$ and
$d((t,0),(t,1))\le A$ for some constant $A$. Question: Is $(X,d)$ isometric to a flat strip
$\R\times [0,D]$? We do not know the answer even under the additional assumption that the translation
map $T:X\to X$, $\,T(t,s)=(t+1,s)$ is an isometry.
In the case of $\CAT(0)$ spaces the flat strip theorem follows from the fact that the
distance function $t\mapsto d(c(t),c'(t))$ is convex for geodesics $c,c':[a,b] \to X$.
We also remark that the convexity of the function $t\mapsto d(c(t),c'(t))$ is the essential tool
to obtain the analogous results of Theorem \ref{thm:cryst} and Theorem \ref{thm:manifold} in
the $\CAT(0)$-case. 
Since we do not have this convexity, our proof in the Ptolemy case is very different.

We outline our main ideas:
Our proof relies on a weaker form of convexity. In geodesic Ptolemy metric spaces 
the distance function $d(p,.)$ to a point $p\in X$ is
convex. As a consequence also Busemann functions are convex.
We show (in the setting of Theorem \ref{thm:manifold})
that the Busemann function of an axis of an element $\al \in\Ga^*\subset \Ga$
(where $\Ga^*$ is the maximal abelian subgroup of $\Ga$)
is indeed an affine function. 
We can show that there are enough affine functions to separate points.
Thus we can apply the main theorem of \cite{HL} to obtain that $X$ is isometric to
a convex subset of a normed vector space. Finally we have to use the
Theorem of Schoenberg \cite{Sch} to see that $X$ is isometric to a Euclidean space.
In the setting of Theorem \ref{thm:cryst} we need suitable modifications.

We should mention that the result and parts of its proof have some analogy to the
"Hopf conjecture" and its proof in \cite{BI}. Compare in this context also \cite{Bu} and the
discussion in \cite{G},p.43 ff.


\subsection{ Motivation and outlook} \label{subsec:motivation}

At the end of this introduction we give an idea why we investigate
geodesic Ptolemy spaces. Our motivation was the study of the boundary of
$\CAT(-1)$ spaces.
On the ideal boundary $\partial_{\infty} Y$ of a $\CAT(-1)$-space $Y$,
Bourdon \cite{B} introduced a distance function
$\rho_o$, depending on a point $o\in Y$. A similar metric $\rho_{\omega}$
on $\d_{\infty} Y \setminus \{\om\}$
was introduced
by Hamenst\"adt \cite{Ha}, where the metric depends on a point $\om \in \d_{\infty}Y$ and
a Busemann function for this point.
The metrics $\rho_o, \rho_{\om}$ are all M\"obius equivalent.
For a detailed description compare \cite{FS}.

Let us describe the elements of M\"obius geometry on metric spaces.
Let $X$ be a metric space, where $|xy|$ is the distance of $x$ and $y$.
To a quadruple $(x,y,z,w)\in X^4\setminus D$ (where $D$ is the diagonal of points where one entry appears
3 or 4 times), we associate the {\em cross ratio triple}
$$\crt(x,y,z,w)= (|xy|\,|zw|:|xz|\,|yw|:|xw|\,|yz|) \in \R P^2.$$
Similar to the classical crossratio there are six reasonable definitions by permuting the entries.
A map $f:X\to Y$ is called a {\em M\"obius map}, if it is injective and for all quadruples from
$X^4 \setminus D$ we have
$\crt(f(x),f(y),f(z),f(w))=\crt(x,y,z,w)$. This definition 
is equivalent to the usual definition using the
classical cross ratio.
Let $\Delta \subset \R P^2$ be the subset of points $(a:b:c)$, such that the three numbers
$|a|,|b|,|c| $ satisfy the triangle inequality.
Then $X$ is a Ptolemy metric space, if and only if for all allowed quadruples
$\crt(x,y,z,w)\in \Delta$.
Therefore the Ptolemy condition is in a natural way a M\"obius invariant condition.
In a certain sense the Ptolemy condition is the natural condition for
involutive geometry: note that a metric space $X$ is Ptolemy if and only if for
all points $x\in X$ the map $\rho_x:X\setminus \{x\} \times X\setminus \{x\} \to \R$,
$\rho_x(y,z)=\frac{|yz|}{|yx|\,|zx|}$ is a metric (see also \cite{BFW}).

In \cite{FS} we made the observation that the Bourdon and Hamenst\"adt metrics on
$\d_{\infty} Y$ are Ptolemy for $\CAT(-1)$-spaces $Y$ and the isometries 
of $Y$ extend to
M\"obius maps on $\d_{\infty} Y$.
In Ptolemy metric spaces one can naturally define the notion
of a circle (in the generalized sense, i.e. a circle or a line).
Recall that
the classical Ptolemy theorem states that four points $x,y,z,w$ (in the
Euclidean plane or more generally in $\R^n$) lie
on a circle (or a line) if and only if
$\crt(x,y,z,w) \in \d \De$, where $\d\De$ is the set of points
$(a:b:c)$ with equality in the triangle inequality.
Thus let us define

\begin{definition}
Let $X$ be a Ptolemy space. A {\em circle} in
$X$ is a continuous map
$f:S^1\to X$, such that for all
$x,y,z,w \in S^1$
we have equality in the Ptolemy inequality, i.e.
$\crt(f(x),f(y),f(z),f(w)) \in \d \De$.
\end{definition}

It then turns out that for points
$x,y,z,w \in S^1$,
which are ordered cyclically on $S^1$
(i.e. $y$ and $w$ lie in different components
of $S^1\setminus\{x,z\}$), we have equality of the form
$|f(x)f(y)|\,|f(z)f(w)| +|f(y)f(z)|\,|f(w)f(x)|=|f(x)f(z)|\,|f(w)f(y)|$.
Note that this notion of a circle is obviously M\"obius invariant and
thus circles are preserved by M\"obius maps.

Now one can reformulate a result of Bourdon (see \cite{B}) in the following way:
Let $Y$ be a $\CAT(-1)$-space and 
consider $\d_{\infty} Y$ as a Ptolemy space with some Bourdon or
Hamenst\"adt metric. Let $f:S^1\to \d_{\infty} Y$ be a circle,
then there exists a totally geodesic subspace $Y_0\subset Y$ isometric to
$\Hy^2$, such that $\d_{\infty} Y_0$ is the image of $f$. Indeed the 
circles in $\d_{\infty}Y$
correspond $1-1$ to the boundaries of 
isometrically embedded hyperbolic planes in $Y$.

This is one motivation to study Ptolemy spaces with many circles.
If we look to our definition of a circle, a circle through infinity
should just be a proper map $f:\R \to X$ such that
for all allowed quadruples $x,y,z,w \in \R$ we have
$\crt(f(x),f(y),f(z),f(w)) \in \d \De$.
For arbitrary but fixed $x,y,z$ and $w \to \infty$ this implies that
$(|f(x)f(y)|:|f(x)f(z)|:|f(y)f(z)|) \in \d \De$.
Equivalently, this condition says that
$f:\R \to X$ is a geodesic line (not necessarily parameterized proportionally to arclength).
From this viewpoint it is also natural to study Ptolemy spaces with
many lines.



\section{Convexity} \label{sec:convexity}

Let $X$ be a metric space.
By $|xy|$ we denote the distance between points
$x,y \in X$.
Unless otherwise stated, we will parameterize geodesics proportionally to arclength. Thus
a geodesic in $X$ is a map
$c:I\to X$ with
$|c(t)c(s)|= \la |t-s|$ for all $s,t \in I$ and some constant $\la \ge 0$.
A metric space is called geodesic if every pair of points can be joined by a geodesic.
In the sequal $X$ will always denote a geodesic metric space.
For $x,y \in X$ we denote by
$m(x,y)=\{z\in X \mid\, |xz|=|zy|=\frac{1}{2}|xy| \}$ the set of midpoints of $x$ and $y$.
A subset $C\sub X$ is {\em convex}, if for $x,y \in C$ also
$m(x,y)\sub C$.

A function $f:X \to \R$ is {\em convex} (resp. {\em affine}), if for all
geodesics $c:I\to X$ the map
$f\circ c:I \to \R$ is convex (resp. affine).

The space $X$ is called {\em distance convex} if for all
$p\in X$ the distance function
$d_p=|\cdot \ p|$ to the point $p$ is convex.
It is called {\em strictly distance convex}, if the functions $t \mapsto (d_p\circ c)(t)$ are 
strictly convex whenever $c:I \to X$ is a geodesic with
$|c(t)\,c(s)|>||p\,c(t)|-|p\,c(s)||$ for all $s,t\in I$.
This  definition is natural, since
the restriction of $d_p$ to a geodesic segment containing $p$ is never strictly convex.
The Ptolemy property easily implies (see \cite{FLS}): 

\begin{lemma} \label{lem-distance-convex}
A geodesic, Ptolemy metric space is distance convex.
\end{lemma}

As a consequence, we obtain that for Ptolemy metric spaces local geodesics
are geodesics. Here we call a map $c:I\to X$ a {\em local geodesic}, if
for all $t\in I$ there exists a neighborhood $t\in I'\subset I$, such that
$c_{|I'}$ is a geodesic.

\begin{lemma} \label{lem:localglobalgeodesic}
If $X$ is distance convex, then every local geodesic
is globally minimizing.
\end{lemma}

\begin{proof}
Assume the contrary, then there is a local
geodesic $c:[0,b] \to X$,
such that there exists $a \in (0,b)$ such that for $p= c(0)$
we have
$|p\,c(a-\ep)| = a- \ep$ for all $\ep \ge 0$ but $|p\,c(a+\ep)| < a+\ep$ 
for $\ep > 0$. The distance convexity applied to $d_p$ and
the minimizing geodesic $c_{|[a-\ep,a+\ep]}$ implies
that $|pc(a)| < a$, a contradiction.
\end{proof}

In \cite{FLS} we gave examples of Ptolemy metric spaces which are not strictly
distance convex.
However, if the space is proper, then the situation is completely different. \\

\noindent {\bf Theorem} \ref{thm:sdc}. 
{\it A proper, geodesic Ptolemy metric space is strictly distance convex.} \\

Before we prove this result, we derive some consequences. The first corollary
has already been proved in \cite{FLS}.

\begin{corollary} \label{cor:uniquemidpoint} (see also \cite{FLS})
Let $X$ be a proper, geodesic Ptolemy space.
Then for $x,y \in X$ there exists a unique midpoint
$m(x,y)\in X$.
The midpoint function
$m:X\times X\to X$ is continuous.
\end{corollary}

\begin{proof}
Assume that there are 
midpoints
$m_1,m_2 \in m(x,y)$. Let
$m$ be a midpoint of $m_1$ and $m_2$.
If $m_1\ne m_2$, then the strict distance convexity according to Theorem \ref{thm:sdc}
implies
$|x m| < |x m_i|=\frac{1}{2}|xy|$, and
$|y m| < |y m_i|=\frac{1}{2}|xy|$, which is a contradiction.
The uniqueness of the midpoint function together with the properness
easily imply the continuity of $m$. 
\end{proof}

\begin{corollary} \label{cor:center}
Let $X$ be a proper, geodesic Ptolemy space, and let $K \subset X$ be a compact 
subset. Then there exists a unique closed ball $B_r(p)$ of minimal radius, such that
$K \subset B_r(p)$.
\end{corollary}

\begin{proof}
The existence of such a ball of minimal radius is clear. We have to prove the uniqueness.
Assume that $K \subset B_r(p_1)\cap B_r(p_2)$, with $p_1\ne p_2$.
Let $m$ be the midpoint of $p_1$ and $p_2$. 
If $q \in K$, then
$|qp_i| \le r$ and the strict convexity of the
distance function implies
$|qm| < r$. By compactness $K\subset B_{r'}(m)$ for some $r'<r$ contradicting
the minimality of $r$.
\end{proof}

\begin{corollary} \label{cor:projection}
Let $X$ be a proper, geodesic Ptolemy space, and $A \subset X$ be a closed and
convex subset.
Then there exists a continuous
projection
$\pi:X\to A$.
\end{corollary}

\begin{proof}
The strict convexity of the distance function easily implies
with a similar argument as in the proof above, that
for given $x\in X$, there exists a unique $a\in A$, with $|ax|$ minimal.
One easily checks that this projection is continuous.
\end{proof}

\begin{remark}
For $\CAT(0)$ spaces this projection is always $1$-Lipschitz.
We do not know if $\pi$ is
$1$-Lipschitz for general proper geodesic Ptolemy spaces.
\end{remark}

The strict convexity of the distance function
together with the properness implies
easily

\begin{corollary}
Let $X$ be a proper, geodesic Ptolemy
space and let $x,y \in X$.
Then there exists a unique geodesic
$c_{xy}:[0,1]\to X$ from $x$ to $y$ and
the map
$X\times X\times [0,1] \to X$,
$(x,y,t)\mapsto c_{xy}(t)$ is continuous.
\end{corollary}

As a consequence $X$ is contractible.
Indeed for the choice of a basepoint
$o\in X$ the map
$\rho:X\times [0,1] \to X$, $\rho_t(x)=c_{ox}(t)$
is a contraction.

\begin{corollary} \label{cor:extgeod}
Let $X$ be a complete, geodesic Ptolemy space.
If $X$ is a topological manifold,
then every geodesic is extendable.
Thus for every geodesic segment
$c:I\to X$ there exists a geodesic
$\bar c :\R \to X$ with
$\bar c_{|I}=c$.
\end{corollary}

\begin{proof}
Let $c:I \to X$ be a geodesic defined on some
interval $I \subset \R$,
which can not be extended.
We have to show that $I =\R$.
If $c$ is a constant map, then it can be extended trivially.
Thus
$I$ consists of more than one point and $c$
is not a constant geodesic.
Since
$X$ is complete, we see that $I$ is (after obvious reparameterization)
of the form $[0,a]$ with $a > 0$ or of the form
$(-\infty,a]$ for some $a>0$. In both cases the geodesic
$c_{|[0,a]}:[0,a] \to X$ can not be extended
over its endpoint $c(a)$, since otherwise also $c$ could be extended by 
Lemma \ref{lem:localglobalgeodesic}.
Let $o=c(0)$ and $x=c(a)$.
The nonextendability of $c_{|0,a]}$ implies that for every point
$y \in X\setminus \{x\}$, the point $x$ is not contained
in $c_{oy}([0,1])$.
Thus $\rho$ defines also a retraction
$\rho_{|X\setminus\{x\}}:X\setminus\{x\} \times [0,1] \to X\setminus\{x\}$,
and hence $X$ as well as $X\setminus\{x\}$ are contractible.
This is in contradiction to the fact that
$H_n(X,\{x\},\Z)=\Z$, since $X$ is a manifold.

\end{proof}

We now prove Theorem \ref{thm:sdc}.
We start with a simple observation.
\begin{lemma} \label{lem-aux}
If $\al,\beta, a,b$ are real numbers with
$\al + \beta, a,  b \ge 0$, then
\begin{equation} \label{eq:generalformula}
\al b +\beta  a \ge (\al+ \beta)\min\{ a, b\}-
| a - b| \min\{|\al|,| \beta|\}.
\end{equation}
\end{lemma}
\begin{proof}
We assume without loss of generality, that
$ a \le  b$.
Then $ b =  a + c$, where $ c =|a-b| \ge 0$. We compute
$\al b +\beta  a = (\al +\beta)a + c \al.$
If $\al \ge 0$, then our estimate holds trivially.
If $\al < 0$, then the assumption implies that
$|\al| =\min\{|\al|,| \beta|\}$ and 
$c\al = -|a-b|\min\{|\al|,| \beta|\}$.
\end{proof}

{\bf Proof of Theorem \ref{thm:sdc}:}
Note that in order to prove Theorem \ref{thm:sdc} it suffices to prove the \\
{\bf Claim:} {\it Let $X$ be a 
proper, geodesic, Ptolemy metric space. Let further $p,x,y\in X$ be points
with 
\begin{equation} \label{eqn-excess}
\Big| |px| \, - \, |py|\Big| \; < \; |xy| .
\end{equation}
Then there exists a midpoint $m$ of $x$ and $y$ such that $|pm|<\frac{1}{2}[|px|+|py|]$.} \\

Indeed, since a geodesic Ptolemy metric space is distance convex by Lemma \ref{lem-distance-convex}, 
the claim implies the uniqueness of midpoints and therefore the validity of Theorem \ref{thm:sdc}. \\

In order to prove the claim, we set 
\begin{displaymath}
\eta \; := \; \frac{||px|-|py||}{|xy|} .
\end{displaymath}
Thus, $0\le \eta <1$ by assumption (\ref{eqn-excess}). 
Let $s:=\frac{1}{2}[|px|+|py|]$ and let $\tau :=\frac{|px|}{s}$
as well as $\rho := \frac{|py|}{s}$. Thus, $0<\tau ,\rho <2$ and $\tau +\rho =2$. \\

We first make the additional assumption that $\tau$ and $\rho$ are very close together, i.e., 
both are very close to $1$.
Namely we assume that 
\begin{equation} \label{eqn-assumption-tau-rho}
\min \Big\{ \frac{1}{\tau},\frac{1}{\rho}\Big\} \; - \; \eta \, \frac{\min \{ \tau ,\rho \}}{\tau \, \rho} \; \ge \;
\frac{1}{2} (1-\eta) .
\end{equation}
Clearly this assumption holds for $\rho = 1 = \tau$ (and hence $\eta=0$), 
and at first reading one might as well consider this special case.
We will now prove the claim under the Assumption (\ref{eqn-assumption-tau-rho}). 
Later we will reduce the general case to 
this special situation.\\

Choose geodesics $\gamma_{px},\gamma_{py}:[0,s]\longrightarrow X$ connecting $p$ to $x$ and $y$, parameterized
proportionally to arclength. For every $t\in[0,s]$ let $\rho_{s-t}:[0,2]\longrightarrow X$ be a geodesic connecting 
$\gamma_{px}(s-t)$ to $\gamma_{py}(s-t)$, parameterized proportionally to arclength. Since $X$ is proper, there exists 
a sequence $t_n\in (0,s)$, $t_n\longrightarrow 0$, such that $\sigma_{s-t_n}(1)\longrightarrow m$, where $m$ is some
midpoint of $x$ and $y$. \\
The distance convexity of geodesic, Ptolemy metric spaces implies $|pm|\le \frac{1}{2}[|px|+|py|]$.
To prove our claim, we assume equality and will finally reach a contradiction. \\
Thus we assume $|pm|=\frac{1}{2}[|px|+|py|]=s$.

\begin{figure}[htbp]
\centering
\psfrag{p}{$p$}
\psfrag{x}{$x$} 
\psfrag{y}{$y$} 
\psfrag{xn}{$x_n$}
\psfrag{yn}{$y_n$}
\psfrag{l/2}{${\scriptstyle l/2}$}
\psfrag{ln/2}{${\scriptstyle l_n/2}$}
\psfrag{phin}{$\phi_n$}
\psfrag{m}{$m$}
\psfrag{mn}{$m_n$}
\psfrag{ttn}{$\tau t_n$}
\psfrag{rtn}{$\rho t_n$}
\psfrag{an-}{${\scriptscriptstyle a_n^-}$}
\psfrag{an+}{${\scriptscriptstyle a_n^+}$}
\psfrag{bn-}{${\scriptscriptstyle b_n^-}$}
\psfrag{bn+}{${\scriptscriptstyle b_n^+}$}
\psfrag{gpx}{$\tau s_n$}
\psfrag{gpy}{$\rho s_n$}
\includegraphics[width=0.9\columnwidth]{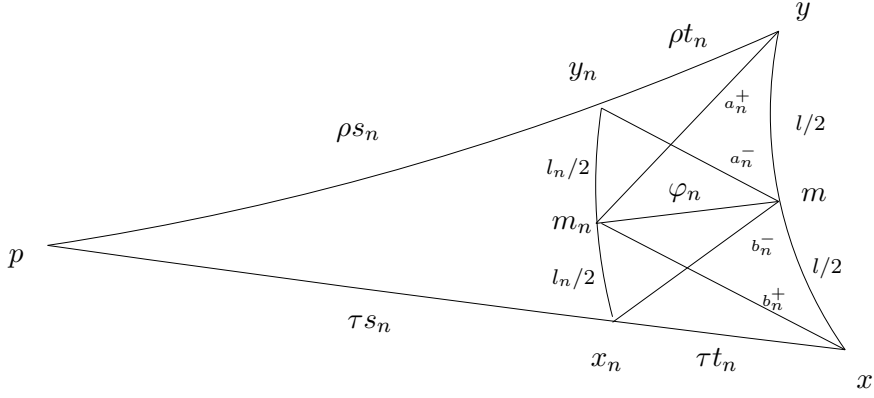}
\caption{The figure visualizes the notation in the proof of Theorem \ref{thm:sdc}.}
\end{figure}

We set
\begin{displaymath}
\begin{array}{lll}
s_n \; := \; s-t_n , & x_n \; := \; \gamma_{px}(s_n), & y_n \; := \; \gamma_{py}(s_n),  \\
m_n \; := \; \sigma_{s_n}(1), & l_n \; := \; |x_ny_n|, & l \; := \; |xy|, \\
b_n^- \; := \; |x_nm|, & a_n^- \; := \; |y_nm|, & b_n^+ \; := \; |xm_n|, \\
a_n^+ \; := \; |ym_n| & \mbox{and} & \varphi_n \; := \; |m_nm|.
\end{array}
\end{displaymath}
Since $m_n\to m$ we have
$\varphi_n \stackrel{n\to \infty}{\longrightarrow}0$. \\
By simple triangle inequality, the Ptolemy inequality for $\{ p,x_n,m,y_n\}$ and the sum of the Ptolemy
inequalities for $\{ x_n,x,m,m_n\}$ and $\{ m_n,m,y,y_n\}$, we obtain the inequalities
\begin{displaymath}
\begin{array}{rrcl}
(I) & a_n^+ \; + \; b_n^+ & \ge & l , \\ 
& & & \\
(II) & a_n^- \tau s_n \; + \; b_n^- \rho s_n & \ge & l_n (s_n+t_n) \;\;\; \mbox{as well as} \\
& & & \\
(III) & a_n^+ a_n^- \: + \; b_n^+ b_n^- & \le & \frac{ll_n}{2} \: + \; 2 t_n\varphi_n .
\end{array}
\end{displaymath}
Now the distance convexity \ref{lem-distance-convex} yields
\begin{displaymath}
b_n^- \; \le \; \frac{1}{2}(\tau\,t_n\;+\;|x_n\,y|)\;\le\;
 \frac{1}{2}(\tau\,t_n\;+\;\rho\,t_n\;+\;l_n)\;=\;
\frac{1}{2}l_n + t_n
\end{displaymath}
and similiarly $a_n^- \le \frac{1}{2}l_n +t_n$. These inequalities together with the triangle inequality 
$a_n^- +b_n^- \ge l_n$ also yield $b_n^-\ge \frac{1}{2}l_n-t_n$ as well as $a_n^-\ge \frac{1}{2}l_n-t_n$.
Hence we obtain $|a_n^- -\frac{l_n}{2}|\le t_n$ and $|b_n^- -\frac{l_n}{2}|\le t_n$. \\
In the same way, we deduce $|a_n^+ -\frac{l}{2}|\le t_n$ and $|b_n^+ -\frac{l}{2}|\le t_n$. \\
Thus, passing to appropiate subsequences, we find $A^-,B^-,A^+,B^+ \in (-2,2)$, such that 
\begin{displaymath}
\begin{array}{rclcrcl}
a_n^- & = & \frac{l_n}{2} \, + \, \frac{A^-}{\tau} t_n \, + \, o(t_n), & \hbox{ }&
b_n^- & = & \frac{l_n}{2} \, + \, \frac{B^-}{\rho} t_n \, + \, o(t_n), \\
& & & & & & \\
a_n^+ & = & \frac{l}{2} \, + \, \frac{A^+}{\tau} t_n \, + \, o(t_n), & \hbox{ } &
b_n^+ & = & \frac{l}{2} \, + \, \frac{B^+}{\rho} t_n \, + \, o(t_n),
\end{array}
\end{displaymath}
and 
\begin{equation} \label{eqn-frac-betraege}
|\frac{A^-}{\tau}|, \; |\frac{B^-}{\rho}|, \; |\frac{A^+}{\tau}|, \; |\frac{B^+}{\rho}| \; \le \; 1.
\end{equation}
The inequalities $(I)$, $(II)$ and $(III)$ above now yield

\begin{displaymath}
\begin{array}{rrcl}
(I') & \frac{A^+}{\tau} \, + \, \frac{B^+}{\rho} & \ge & 0 , \\ 
& & & \\
(II') & A^- \; + B^- & \ge & \frac{l}{s} \;\;\; \mbox{as well as} \\
& & & \\
(III') & \frac{A^-}{\tau} \, + \, \frac{A^+}{\tau} \, + \, \frac{B^-}{\rho} \, + \, \frac{B^+}{\rho} & \le & 0.
\end{array}
\end{displaymath}
But Lemma \ref{lem-aux}, the inequalities (\ref{eqn-frac-betraege}), inequality $(II')$, $|\tau -\rho |=\frac{l}{s}\eta$ 
and inequality (\ref{eqn-assumption-tau-rho}) also yield
\begin{eqnarray*}
\frac{A^-}{\tau} \, + \, \frac{B^-}{\rho} & \ge &
(A^- +B^-) \min \Big\{ \frac{1}{\tau} ,\frac{1}{\rho}\Big\} \; - \; 
\Big| \frac{1}{\rho} - \frac{1}{\tau}\Big| \, \min \Big\{ |A^-|,|B^-|\Big\} \\
& \ge & \frac{l}{s} \min \Big\{ \frac{1}{\tau} ,\frac{1}{\rho}\Big\} \; - \; 
\frac{|\tau -\rho|}{\rho \tau} \,  \min \Big\{ \tau ,\rho \Big\} \\
& \ge & \frac{l}{s} \frac{1}{2} (1-\eta ) \; > \; 0, 
\end{eqnarray*}
which contradicts to the simultanious validity of the inequalies $(I')$, and $(III')$. \\

Finally we have to get rid of the assumption
that $\rho$ and $\tau$ are close together.
Consider some arbitrary
$p,y_0,y_2$ with
$\eta:= \frac{|\,|py_0|\,-\,|py_2|\,|}{|y_0\,y_2|}< 1$.
Consider a geodesic $[0,2] \to X$,
$t \mapsto y_t$ from $y_0$ to $y_2$ parameterized
proportionally to arclength.
If $|p\,y_1| < s_0=\frac{1}{2}(\,|py_0|+|py_2|\,)$, we are done.
Thus we assume
$|p\,y_1|=s_0$ which then implies by Lemma \ref{lem-distance-convex}
that $t\mapsto |p\,y_t|$ is affine.
As a consequence we see that 
$$\frac{|\,|p\,y_{1-t}|-|p\,y_{1+t}|\,|}{|y_{1-t}\,y_{1+t}|} = \eta$$
is constant.
Thus if we take as a new triple the points
$p,y_{1-t},y_{1+t}$ for some $t> 0$,
then the corresponding $\eta$ is the same
as for the original triple.
Now we can choose $t>0$ sufficiently small,
such that
for
$y_0':=y_{1-t},y_2':=y_{1+t}$
and
$\rho' s_0=|py_0'|$,
$\tau' s_0=|py_2'|$
the estimate
(\ref{eqn-assumption-tau-rho})
holds
(this is possible, since $\eta$ is fixed and
$\rho',\tau'$ tend to 1 for $t\to 0$).
Thus by the above argument,
there exists
a midpoint
$y_1'$ of $y_0'$ and $y_2'$ with
$|p\,y_1'|< s_0$.
Clearly
$y_1'$ is also a midpoint of
$y_0$ and $y_2$ and we are done.
\hfill $\Box$



\section{Normed Vector Spaces} \label{sec:nvs}

We recall some facts about normed vector spaces
which we will use in the sequel.

Let $(V,\|\cdot \|)$ be a normed vector space.
A line in $V$ is a map
$c:\R\to V$ of the form
$c(t)=a+tb$, where $b \neq 0$. A line is always a geodesic,
however if the normball is not strictly convex, there are
other geodesics.
We call a subset $C \subset V$ {\em line-convex}, if for
$x,y \in C$ also the segment $[0,1]\to V$,
$t\mapsto ty+(1-t)x$ is contained in $C$.
We use the notion line-convex since it differs from the notion
of convexity as defined in Section \ref{sec:convexity} for normed spaces
with non strictly (line)-convex normball.

Let $\ga \in \ISO (V)$, then a line $c:\R\to V$ parameterized by arclength 
is called an axis of $\ga$, if $\ga c(t) = c(t+L)$ for some
$L > 0$. If $\ga$ has an axis, then we
denote with $\Ax (\ga)\subset V$ the set of points, which lie on
axes of $\ga$.
Then $\Ax(\ga)$ is a closed, line-convex subset of $V$.
If $\al \in \ISO (X)$ commutes with $\ga$,
then
$\Ax(\ga)$ is $\al$-invariant.

Essential for our argument is the following result of
Schoenberg \cite{Sch}

\begin{theorem} \label{thm:schoenberg}
A normed vector space $(V,\|\cdot \|)$ is a Ptolemy space,
if and only if it is an inner product space.
\end{theorem}

Using asymptotic cones, it is not difficult to prove

\begin{lemma}
If a normed vector space is rough isometric to
a Ptolemy space, then it is a Ptolemy space.
\end{lemma}

Therefore we obtain

\begin{corollary} \label{cor:corschoenberg}
If a normed vector space $(V,\|\cdot \|)$ is rough isometric to a
Ptolemy space, then it is an inner product space.
\end{corollary}


\section{Busemann functions}

In this section $X$ allways denotes a geodesic Ptolemy space.
Let $c:[0,\infty) \to X$ be a geodesic ray parameterized by arclength.
As usual we define the {\em Busemann function}
$b_c(x) = \lim_{t\to\infty}(|xc(t)|-t)$.
Since $b_c$ is the limit of the convex functions $d_{c(t)} -t$,
it is convex.

The following proposition implies that, in a Ptolemy space, rays 
with sublinear distance to each other define (up to
a constant) the same Busemann functions.

\begin{proposition} \label{prop:sublinearrays}
Let $X$ be a Ptolemy space, let $c_1,c_2:[0,\infty) \to X$ be rays with
Busemann functions $b_i:=b_{c_i}$, and assume
$\lim_{t\to \infty}\frac{1}{t}|c_1(t)c_2(t)|=0$.
Then 
$(b_1-b_2)$ is constant.
\end{proposition}

\begin{proof}
Let $p,q\in X$ and let
$P_i(t)=|pc_i(t)|$,
$Q_i(t)=|qc_i(t)|$.
We define $P_i'(t), Q_i'(t)$, such that
$P_i(t)=t+P_i'(t)$ and
$Q_i(t)=t+Q_i'(t)$.
Hence we have
$b_i(p)=\lim P_i'(t)$ and
$b_i(q)=\lim Q_i'(t)$.
Since $X$ is Ptolemy, the three numbers
$P_1(t)Q_2(t),P_2(t)Q_1(t),|c_1(t)c_2(t)|\,|pq|$
satisfy the triangle inequality.
The assumption 
$\lim_{t\to \infty}\frac{1}{t}|c_1(t)c_2(t)|=0$
now implies
$$\lim_{t\to \infty} \frac{1}{t}(P_1(t)Q_2(t)-P_2(t)Q_1(t)) = 0$$
and hence
$$\lim_{t\to \infty} (P_1'(t)+Q_2'(t)-P_2'(t)-Q_1'(t)) =0.$$
The last limit implies
$(b_1-b_2)(p)=(b_1-b_2)(q)$.

\end{proof}

Let now $c:\R \to X$ be a geodesic line parameterized by arclength.
Let $c^\pm:[0,\infty)\to X$ be the rays
$c^+(t)=c(t)$ and $c^-(t)=c(-t)$.
Let further $b^\pm:=b_{c^\pm}$.

The following follows easily.

\begin{lemma} \label{lem:pos}
$(b^++b^-)\ge 0$ and
$(b^++b^-)=0$ on the line $c$.

\end{lemma}

We now consider these Busemann functions for lines
with sublinear distance.

\begin{proposition} \label{prop:sublinearlines}
Let $c_1,c_2:\R\to X$ be lines with
$\lim_{t\to \infty} \frac{1}{t}|c_1(t)c_2(t)| =0$ 
and
$\lim_{t\to \infty} \frac{1}{t}|c_1(-t)c_2(-t)| = 0$.
Then
$(b_1^++b_1^-)=(b_2^++b_2^-)$.
\end{proposition}

\begin{proof}
By Proposition \ref{prop:sublinearrays} we have that
$(b_1^+-b_2^+)$ and $(b_1^--b_2^-)$ are constant which implies that
$(b_1^++b_1^--b_2^+-b_2^-)=a$ for some constant $a$.
We show that $a=0$.
If $p_1\in c_1(\R)$, then by Lemma \ref{lem:pos}
$(b_1^++b_1^-)(p_1) =0$ and $(b_2^++b_2^-)(p_1)\ge 0$.
Thus $a\le 0$.
If $p_2 \in c_2(\R)$, then
$(b_2^++b_2^-)(p_2) = 0$ and
$(b_1^++b_1^-)(p_2) \ge 0$. Thus $a\ge 0$.
\end{proof}


\section{Groups operating on $X$}

In this section $X$ always denotes a geodesic
Ptolemy metric space.
Let $\Ga \subset \ISO(X)$ be a subgroup 
of the isometry group operating properly and cocompactly
on $X$.
As already stated in the introduction, this implies that $X$ is
proper and in particular complete.
An element $\ga \in \Ga$ is a {\em torsion element}, if
$\ga^m = \id$ for some $m \ge 1$.
To an element
$\ga \in \Ga$ we associate the {\em displacement function}
$d_{\ga}(x):=|x\,\ga x|$.
A geodesic line is called an axis of an element
$\ga$, if there exists some $L>0$ such
that $\ga c(t) = c(t+L)$ for all $t \in \R$.
We always parametrize lines by arclength.
By $\Ax(\ga):=\{x\in X | x \ \mbox{lies on an axis of}\  \ga\}$
we denote the union of all axes of $\ga$. 

\begin{lemma}
(1) $d_{\ga}$ assumes the minimum $\min d_{\ga}\ge 0$.

(2) $\min d_{\ga} =0$ if and only if $\ga$ is a torsion element.

(3) If $\ga$ is not torsion,
then $\Ax(\ga)=\{x\in X | d_{\ga}(x) \ \mbox{is minimal}\ \}$.

(4) If $c_1,c_2$ are axis of $\ga$, then $|c_1(t)c_2(t)|$ is bounded.

(5) If $\al$ commutes with $\ga$ and $c$ is an axis of $\ga$, then also
$\al c$ is an axis of $\ga$.

\end{lemma}

\begin{proof}
(1) is standard. See e.g. the proof of Lemma 2.1. in \cite{CS}.

(2) If $\min d_{\ga}=0$, then $\ga$ has a fixed point. Since
$\Ga$ operates properly, $\ga$ is a torsion element.
If $\ga$ is torsion, then, for given $x \in X$, the orbit
$K=\{\ga^m (x)| m\in \Z\}$ is a $\ga$-invariant compact set.
By Corollary \ref{cor:center}, there exists a unique
$p\in X$, such that $K \subset B_r(p)$, where $r$ is minimal.
Then $p$ is fixed by $\ga$, hence $\min d_{\ga}=0$.

(3) Let $\ga$ be a nontorsion element.
Let $L =\min d_{\ga} >0$ be the minimum of the 
displacement function. Let $x\in X$ with $d_{\ga}(x) =L$.
Let $c:\R \to X$ be the piecewise geodesic with
$c(nL)=\ga^n(x)$.
Note that
$c(\frac{2n+1}{2}L)=m(\ga^n(x),\ga^{n+1}(x))$ 
and hence $\ga c((\frac{2n-1}{2})L)=c((\frac{2n+1}{2})L))$.
Thus
$$d_{\ga}(c((\frac{2n-1}{2})L))\le |c(\frac{2n-1}{2}L) c(nL)|+ |c(nL)c(\frac{2n+1}{2}L)|=L.$$
By minimality of $L$ we have equality which implies that
$c_{|[\frac{2n-1}{2}L,\frac{2n+1}{2}L]}$ is minimizing.
Thus since also $c_{|[nL,(n+1)L]}$ is minimizing, $c$ is locally minimizing
and hence a geodesic by Lemma \ref{lem:localglobalgeodesic}.
Thus $c$ is an axis of $\ga$.

Assume now that $c:\R \to X$ is an axis with $\ga c(t) =c(t+L)$, $L >0$.
We show that $L =\min d_{\ga}$.
Let therefore $x\in X$ and let $a=d(x,c(0))$.
By triangle inequality
$$mL=|c(0)\ga^m c(0)|\le |c(0)x|+|x\ga^m x|+|\ga^m x \, \ga^m c(0)|\le a+ m\, d_{\ga}(x) +a.$$
Hence $mL\le 2a+ m\, d_{\ga}(x)$ for all $m$, which implies $d_{\ga}(x) \ge L$. 

(4) Since $c_1,c_2$ are axes parameterized by arclength, we have $|c_1(t+L)c_2(t+L)|=|c_1(t)c_2(t)|$.

(5) If $\al$ commutes with $\ga$ and $\ga c(t) =c(t+L)$, then
$\ga (\al c(t))=\al (\ga c(t))=\al c(t+L)$,
hence also $\al\circ c$ is an axis of $\ga$.
\end{proof}

Let now $\ga$ be a nontorsion element of $\Ga$.
We define
$b_{\ga}=b_c$, where $c$ is an axis of $\ga$.
Since two axis have bounded distance,
$b_{\ga}$ is well defined up to a constant by Proposition \ref{prop:sublinearlines}.
More generally we define
$b^{\pm}_{\ga}=b^{\pm}_c$
and
$B_{\ga}=b_{\ga}^++b_{\ga}^-$. 
Note that 
$B_{\ga}$ is
well defined and does not depend on the choice of an axis by
Proposition \ref{prop:sublinearrays}.

\begin{lemma} \label{lem:businvariance}
(1) If $\al$ commutes with $\ga$, then
$b_{\ga}(\al(x))-b_{\ga}(x)$ is constant (in $x$).

(2) $b_{\ga}(\ga(x))-b_{\ga}(x) = \min d_{\ga}$.

(3) $B_{\ga}=0$ on all axes of $\ga$.

(4) $B_{\ga^{-1}}=B_{\ga}$ and 
$B_{\ga}(\al(x))=B_{\al^{-1}\ga\al}(x)$ for
all $\al \in \Ga$. 
\end{lemma}

\begin{proof}
(1) If $\al$ commutes with $\ga$ and $c:\R\to X$
is an axis of $\ga$, then also $\al^{-1}\circ c$ is an axis of $\ga$ and
hence in bounded distance to $c$. Thus
$b_{\al^{-1}\circ c}$ and $b_c$ differ by a constant by Proposition
\ref{prop:sublinearrays}.
Since $b_c(\al x)= b_{\al^{-1}\circ c}(x)$ we otain the result.

(2) Because of (1) we need to compute the equality only
for $x=c(0)$ where $c$ is an axis of $\ga$.
Since $\ga c(0)=c(L)$ with $L=\min d_{\ga}$ we are done.

(3) is merely a reformulation of  Lemma \ref{lem:pos}.

(4) follows from an easy computation.

\end{proof}

Now the proof of Theorem \ref{thm:straight} 
follows from the arguments in Section III of \cite{CS}.
The essential tool for the proof in that paper was Corollary 2.6 which corresponds
exactly to our Lemma \ref{lem:businvariance}. One can easily check
that the proof in the cited paper transfers word by word to the present situation.


\section{Affine functions}

Let $X$ be a geodesic metric space.
Recall that for $x,y \in X$ we denote by
$m(x,y)=\{z\in X \, | \, d(x,z)=d(z,y)=\frac{1}{2}d(x,y)\}$ the set of midpoints of $x$ and $y$.
A map $f:X \to Y$ between two geodesic metric spaces is called
{\em affine}, if
for all $x,y \in X$,
we have
$f(m(x,y)) \subset m(f(x),f(y))$.
Thus a map is affine if and only if it maps
geodesics parameterized proportionally to arclength into 
geodesics parameterized proportionally to arclength.
An affine map $f:X \to \R$ is called an affine function.
Note that this definition coincides with the definition given in Section \ref{sec:convexity}

Given a geodesic metric space $X$, we set
$$\cA'(X):=\{f:X\to \R|\ f\ \mbox{affine and Lipschitz}\}.$$
$\cA'(X)$ is a Banachspace, where $\| f\|$ is the optimal Lipschitz constant.
We set $\cA(X):=\cA'(X)/\sim$, where $f\sim g$ if $(f-g)$ is constant.
By $[f]$ we denote the equivalence class of $f$.
Then also $\cA(X)$ is a Banachspace, where $\|[f]\|=\|f\|$.
The group
$\ISO(X)$ acts on $\cA'(X)$ (and $\cA(X)$) by linear isometries via
$\ga(f)(x)=f(\ga^{-1}(x))$ (resp. $\ga([f])=[\ga(f)]$).
Let $\cA^*(X)$ be the Banach dual space of $\cA(X)$ with the norm
$$\|\rho\|=sup_{[f]\in\cA(X)}\frac{|\rho([f])|}{\|[f]\|}.$$
Also $\ISO(X)$ operates linearly on $\cA^*$ by isometries via
$\ga(\rho)([f])=\rho(\ga^{-1}([f]))$.
For $x,y \in X$ let $E(x,y)\in \cA^*(X)$ 
be the evaluation 
map $E(x,y)([f])=f(x)-f(y)$. 
It follows directly from the definitions
that
$\|E(x,y)\|\le |xy|$.
For $\ga \in \ISO(X)$ we have
$\|E(\ga(x),\ga(y))\|=\|E(x,y)\|$.

For a given basepoint $o\in X$
consider the subset
$\cE_o=\{E(x,o)|x\in X\} \subset \cA^*$,
and 
the map
$$A_o:X\to \cE_o,\ \ \ \ \ \ \ \ x\mapsto E(x,o).$$
If $x,y\in X$ and $z\in m(x,y)$ then for all $[f] \in \cA$ we have
$$E(z,o)[f]=f(z)-f(o)=\frac{1}{2}(f(x)+f(y))-f(o)
=(\frac{1}{2}E(x,o)+\frac{1}{2}E(y,o))[f].$$
Thus $A_o$ is affine and maps the geodesics in $X$ to lines in
$\cE_o \subset \cA^*$. In particular $\cE_o$ is a line-convex subset of $\cA^*$.
Since $\|E(x,o)-E(y,o)\|=\|E(x,y)\|\le |xy|$ we see that
$A_o$ is $1$-Lipschitz.

We consider the map
$a_o:\ISO(X) \to \ISO(\cE_o)$
given by
$a_o(\ga)( E(x,o)):= E(\ga x,o)$. Indeed $a_0(\ga )$ is an affine isometric
operation, since
$$\|E(\ga x,o)-E(\ga y,o)\|=\|E(\ga x,\ga y)\|=\|E(x,y)\|=\|E(x,o)-E(y,o)\|.$$
With respect to this affine operation the map $A_o$ is $\ISO(X)$-equivariant,
i.e. $A_o(\ga x)=a_o(\ga) A_o(x)$.
The isometries $a_o(\ga)$ map lines in $\cE_o$ to lines 
in $\cE_o$.
We remark that this affine action of $\ISO(X)$ on $\cE_o$ is {\em not}
the restriction of the linear action of $\ISO(X)$ on $\cA^*$ to the subset
$\cE_o$. 

We recollect all this information in the following 

\begin{lemma} \label{lem:aff}
If $X$ is a geodesic metric space, then the map
$A_o:X\to \cE_o$ is a surjective, affine, $1$-Lipschitz and $\ISO(X)$-equivariant
map of $X$ onto the line-convex subset $\cE_o\subset \cA^*$. 
\end{lemma}

\begin{definition}
We say that {\em affine functions separate points on $X$},
if for different $x,y \in X$ there exists $f\in \cA'(X)$ with
$f(x)\neq f(y)$.
\end{definition}

Note that affine functions separate points if and only
if the map $A_o$ is injective. In this case we have the following
strong result of Hitzelberger and Lytchak \cite{HL}.

\begin{theorem}
If affine functions separate points, then
$A_o:X\to \cE_o$ is an isometry.
\end{theorem}


\section{Proof of Theorems \ref{thm:cryst}, \ref{thm:manifold} and Corollary \ref{cor:solv}}

In this section we assume that $X$ is a 
geodesic Ptolemy space and $\Ga\subset \ISO(X)$ is a crystallographic group
operating properly and cocompactly on $X$. Recall that 
under these assumptions $X$ is proper.

Recall the following characterization of crystallographic
groups, (see \cite{W}, Theorem 3.2.9).

A group $\Ga$ is crystallographic, if and only if $\Ga$ has a normal, free
abelian subgroup $\Ga^*$ of finite rank and finite index in $\Ga$ which is
maximal abelian in $\Ga$. In that case, $\Ga^*$ is unique.

Note that $\Ga^* \simeq \Z^n$, since $\Ga^*$ is free of finite rank.
Since $\Ga^*$ has finite index, $\Ga^*$ also operates cocompactly on
$X$.
Let $\al \in \Ga^*\setminus \id$ , let $c$ be an axis of
$\al$ and let $b^{\pm}_{\al}=b^{\pm}_c$ be the corresponding Busemann functions.
It follows from Lemma \ref{lem:businvariance} that
the function
$B_{\al}=b^+_{\al}+b^-_{\al}$ is $\Ga^*$ invariant.
Since $\Ga^*$ acts cocompactly the function
$B_{\al}$ is bounded.

\begin{proof}(of Theorem \ref{thm:manifold})
We now assume in addition that $X$ is a manifold.
Thus by Corollary \ref{cor:extgeod} all geodesics in $X$ can be extended.
Thus the bounded convex functions
$B_{\al}$ are constant for all $\al \in \Ga^*\setminus \id$.
Since $B_{\al}(x)=0$ for a point on an axis of $\al$, we
have that $B_{\al}=0$ on $X$.
Thus $B_{\al}=b^+_{\al}+b^-_{\al}$ is in particular affine,
and since the $b^{\pm}_{\al}$ are convex, the functions
$b^{\pm}_{\al}$ are actually affine.

Since $A_o$ maps geodesics in $X$ to lines in $\cE_o$ and since
all geodesics are extendable, the image $\cE_o$ is an affine subspace of
$\cA^*$.

We will use the following fact which follows immediately from the 
cocompactness
of $\Ga^*$: there exists a constant
$C_1>0$, such that for all $p,q \in X$ there exists $\al \in \Ga^*$ with
$|p \al(q)|\le C_1$.

We now show that the affine functions separate points.
Let therefore $x,y \in X$ with $x\ne y$.
We first assume that
$|xy|>6C_1$. Then there exists $\al \in \Ga^*\setminus \id$, such that
$|\al(x)y|\le C_1$. There also exists an axis
$c:\R\to X$ of $\alpha$ such that $|x c(0)|\le C_1$ (this follows since for any axis
$c'$ of $\al$ and any $\ga \in \Ga^*$, also $\ga \circ c'$ is an axis of
$\al$).
Consider now the affine function $b_{\al}$. Since $b_{\alpha}$ is $1$-Lipschitz,
we have
$|b_{\al}(x)-b_{\al}(c(0))|\le C_1$,
$|b_{\al}(\al c(0))-b_{\al}(\al x)|\le C_1$,
$|b_{\al}(\al x)-b_{\al}(y)| \le C_1$, and
$|b_{\al}(\al c(0)) b_{\al}(c(0))|= |c(0)\al(c(0))|$ by 
Lemma \ref{lem:businvariance} (2).
Thus
$$|b_{\al}(x)-b_{\al}(y)|\ge |c(0)\al(c(0))| - 3C_1 \ge |xy|-6C_1 >0,$$
which shows that $x,y$ are separated by $b_{\al}$.

Let now $x,y \in X$ be arbitrary with $x\ne y$.
Let $c:\R \to X$ be the unit speed geodesic with
$c(0)=x$ and $c(\tau)=y$ for $\tau=|xy|$.
Let $s>6C_1$ and $y'=c(s)$.
By the above argument there exists an affine function $f$,
with $f(x)\ne f(y')$. Thus the affine function
$f\circ c$ is not constant, and hence $f(x)\ne f(y)$.

Since affine functions separate points, the map
$A_o:X \to \cE_o$ is an isometry by the Hitzelberger Lytchak Theorem.
Thus $\cE_o$ is an affine subspace of the normed vector space $\cA^*$
which is in addition
a Ptolemy space. By Schoenbergs result it is an inner product space and hence
isometric to the
standard $\R^n$, where $n =\dim X$.
\end{proof}

We now turn to the proof of Theorem \ref{thm:cryst}.
We assume that a crystallographic group operates 
properly and cocompactly 
on the 
geodesic Ptolemy space $X$.

Let $X'\subset X$ be a nonempty, closed, convex and 
$\Ga$-invariant subset,
which is minimal with respect to these
properties. Since $\Ga$ operates 
cocompactly on $X$, it is easy to show, that such 
minimal set $X'$ exists. Namely, consider an arbitrary chain $\{ A_j\}_{j\in J}$  
of such sets. Then this chain is bounded by the nonempty, closed, convex and
$\Gamma$-invariant set $\bigcap A_j$. In order to see that this set is nonempty, 
just take a compact, fundamental region $K$ of $\Gamma$ and consider the intersections
$A_j\cap K$, $j\in J$. Since $K$ is a compact fundamental region, and the $A_j$ are $\Gamma$-invariant,
it follows that  $\bigcap A_j \neq \emptyset$. Now the existence of a minimal set $X'$
follows from Zorn's Lemma.
 
By Corollary \ref{cor:projection}
there exists a continuous projection
$\pi:X \to X'$. This projection is
$\Ga$-equivariant and since $\Ga$ operated cocompactly on
$X$ as well as on $X'$,
there exists a constant $C\ge 0$ such that
$|x\,\pi(x)|\le C$ for all $x\in X$ and hence
$\pi$ is a rough isometry.
  Thus we have to prove the theorem only for $X'$. Thus we can
suppose that $X$ satisfies the following additional

{\bf Assumption:} If $W \subset X$ is a nonempty closed convex and
$\Ga$-invariant
subset, then $W=X$.

\begin{lemma}
For all $x\in X$ and
for all $\al \in \Ga^*\setminus \id$ we have

(1) $B_{\al}(x) = 0$ and hence the functions $b^{\pm}_{\al}$ are affine.

(2) $A_o(x)\in\Ax(a_o(\al))$.

\end{lemma}

\begin{proof}
Let $W \subset X$ be the subset of all $x\in X$ such that
(1) and (2) hold for all $\al \in \Ga^*\setminus \id$. 
We have to show that
$W=X$.

Since $B_{\al}$ is a continuous convex function,
the condition (1) is closed and convex.
Since $Ax(a_o(\al))$ is a closed line-convex subset and
$A_o$ is continuous, also condition (2) is closed and convex.
Since $\Ga^*$ is a normal subgroup of $\Ga$, (1) and (2) describe a 
$\Ga$-invariant condition. Thus $W$ is closed convex and
$\Ga$-invariant.

We now show that $W$ is not empty.
To prove this, let $W' \subset X$ be a nonempty, closed, convex and 
$\Ga^*$-invariant subset,
which is minimal with these
properties. Since $\Ga^*$ operates (as a finite index subgroup of $\Ga$)
cocompactly on $X$, it follows as above, that such 
minimal set $W'$ exists. We show that
$W '\subset W$ and hence $W$ is not empty.
Let $\al \in \Ga^*\setminus \id$. Since
$\Ga^*$ is abelian, the convex function
$B_{\al}$ is $\Ga^*$-invariant and hence
$\{ B_{\al}=0\}$ is closed, convex and $\Ga^*$-invariant.

Now the function $d_{\al}$ assumes a minumum
on the subset $W'$, hence there exists an axis
$c:\R\to W'$ of $\al$.

Thus it follows from 
Lemma \ref{lem:businvariance} (3) that
$W'\cap \{ B_{\al}=0\}\ne \emptyset$.
This shows that $W'\subset \{B_{\al}=0\}$, since $W'$ was assumed to be minimal
and $\{ B_{\al}=0\}$ is closed, convex and $\Ga^*$-invariant. \\
Recall that $A_o$ maps geodesics to
lines or to a point.
We show that $A_o\circ c(t)= E(c(t),o)$ is actually a complete line.
Since $B_{\al}$ is constant and hence affine, also the functions 
$b^{\pm}_{\al}$ are affine. Since $b^{\pm}_{\al}$ are not constant on
$c(t)$, the map $t\mapsto E(c(t),o)$ is also not constant.
It follows that
$A_o\circ c$ is 
a complete line which is $a_o(\al)$-invariant and hence
an axis of $a_o(\al)$.
Since $\Ga^*$ is abelian, the set
$C:=\Ax(a_o(\al))\subset \cE_o$ is closed, line-convex and $a_o(\Ga^*)$-invariant.
Hence $A_o^{-1}(C)$ is closed, convex and $\Ga^*$-invariant.
By minimality $W'\subset A_o^{-1}(C)$, which implies that
for $x \in W'$, we have
$A_o(x)\in\Ax(a_o(\al))$.
Thus we have shown that $W'\subset W$ and hence $W \ne \emptyset$.
By the additional assumption $W=X$.
\end{proof}

\begin{lemma}
The map $A_o:x\mapsto E(x,o)$ is a rough isometry.
\end{lemma}

\begin{proof}

We know already that $A_o$ is $1$-Lipschitz.

We use arguments which are very similar to the arguments
in the proof of Theorem \ref{thm:manifold}.
Again $\Ga^*$ operates cocompactly on $X$, and hence there exists $C_1>0$
such that for all $p,q \in X$ there exists $\al \in \Ga^*$ with
$|p \al(q)|\le C_1$.

Let now $x,y \in X$ with $|xy| > 6 C_1$,
and choose $\al \in \Ga^*$ exactly as in the proof of 
Theorem \ref{thm:manifold}.

Then
$$\|E(x,o)-E(y,o)\|=\|E(x,y\|\ge \|E(x,y)([b_{\al}])\|
=|b_{\al}(x)-b_{\al}(y)|\ge |xy|-6C_1,$$
where the arguments are exactly as above. The claim follows.

\end{proof}

\begin{lemma} \label{lem:cE}
$\cE_o$ is an affine subspace of $\cA^*$ with
 $\dim(\cE_o)=\rank(\Ga^*)$.

\end{lemma}

\begin{proof}

By the above Lemma all $a_o(\al)$, $\al \in \Ga^*\setminus \id$ 
operate as translations
on the set $\cE_o$. Thus the line-convex hull of
$a_0(\Ga^*)E(x,o)$ is an affine subspace $H_x \subset \cE_o$ of
dimension equal to $\rank(\Ga^*)$
for all $x \in X$.
Note that all $H_x$ are parallel to each other.
Since $\Ga^*$ is a finite index normal subgroup of $\Ga$.
there are finitely many points $x_0,\ldots,x_k \in X$, such that
$\Ga$ leaves the set $H_{x_o},\ldots,H_{x_k}$ invariant.
The line-convex hull 
of these subspaces defines a finite dimensional affine subspace
$V$ 
of $\cA^*$, and the group $a_o(\Ga)$ leaves this space invariant and
operates affinely on $V$, i.e. maps lines to lines.
Since $a_o(\Ga)$ also maps parallels of $H_x$ to parallels of $H_x$,
$a_o(\Ga)$ also operates affinely on the quotient space
$V/H_{x_o}$, and has on that space a finite orbit $[x_0],\ldots,[x_k]$.
Therefore this operation has a fixed point
$[v]$ which is contained in the line-convex hull of $[x_0],\ldots,[x_k]$.
This implies that there exists
an affine subspace $H$ parallel to $H_{x_o}$ which is $a_o(\Ga)$-invariant
and which is contained in the line-convex hull of
$H_{x_o},\ldots,H_{x_k}$. Thus $H$ is in the image of $A_o$.

Let $W=A_o^{-1}(H)$. Then $W\subset X$ is a nonempty
convex $\Ga$-invariant subset, and hence $W=X$ by our additional assumption.
It follows that for all $x\in X$, $H_x \subset H$ and by dimension reasons
$H_x=H$. It follows that $\cE_o=H$, which proves the Lemma.
\end{proof}

Thus we see that $A_o:X\to \cE_o$ is a $\Ga$-equivariant rough isometry.
By Lemma \ref{lem:cE} $\cE_o$ is a normed vector space which is
rough isometric to the Ptolemy space $X$.
By the Corollary \ref{cor:corschoenberg} to Schoenbergs result,
we see that $\cE_o$ is an inner product space and hence
we obtain
Theorem \ref{thm:cryst}.

Finally we prove Corollary \ref{cor:solv}:

Let $\Ga$ be a virtual polycyclic group. Let
$\Pi \subset \Ga$ be a polycyclic subgroup of finite index.
By a theorem of Hirsch (see e.g. \cite{R} p.139),
$\Pi$ contains a torsion free solvable subgroup $\Sigma$ of
finite index.
Then also $\Sigma$ operates cocompactly on $X$ and
by Theorem \ref{thm:straight} all abelian subgroups
of $\Sigma$ are straight. 
Note further that all abelian subgroups of $\Sigma$ are finitely generated,
since $\Pi$ is polycyclic.
Now a torsion free solvable group,
such that all abelian subgroups are finitely generated and
straight is a crystallographic group by \cite{CS} Lemma 4.3.
Thus $\Sigma$ and hence also $\Ga$ contains a free abelian subgroup $\Sigma^*$ of finite index
and finite rank. It follows that $\Ga$ has
also a normal free abelian subgroup $\Ga^*$ of finite index and finite rank.
But this fact is the only property of $\Ga$ which we needed in the proof of Theorems
\ref{thm:cryst} and \ref{thm:manifold}.

At the end we pose an 

{\bf Open Question}: Let $\Ga$ be a crystallographic group operating 
properly and cocompactly on a geodesic Ptolemy space $X$. Does there
exist a totally geodesic flat Euclidean subspace $F\subset X$ which
is $\Ga$-invariant?


\bigskip
\begin{tabbing}

Thomas Foertsch,\hskip10em\relax \= Viktor Schroeder,\\ 

Mathematisches Institut,\>
Institut f\"ur Mathematik, \\

Universit\"at Bonn,\> Universit\"at Z\"urich,\\
Beringstr. 1, \>
 Winterthurer Strasse 190, \\

D-53115 Bonn, Deutschland\>  CH-8057 Z\"urich, Switzerland\\

{\tt foertsch@math.uni-bonn.de}\> {\tt vschroed@math.unizh.ch}\\

\end{tabbing}
\end{document}